\documentclass{article}

\usepackage{amsmath,amsfonts,amsthm,amssymb}
\usepackage{appendix}
\usepackage{color}
\usepackage{comment}
\usepackage{enumitem}
\usepackage{graphicx}
\usepackage{multicol}
\usepackage{nicefrac}
\usepackage{setspace}
\usepackage{soul}
\usepackage{tikz}
\usepackage[textwidth=1.5cm,color=mnt!30,linecolor=red,textsize=footnotesize]{todonotes}
\usepackage[normalem]{ulem}
\usepackage{verbatim}
\usepackage{xcolor}

\usepackage{hyperref}

\newtheorem{thm}{Theorem}

\newtheorem{cnj}[thm]{Conjecture}
\newtheorem{cor}[thm]{Corollary}
\newtheorem{exa}[thm]{Example}
\newtheorem{fct}[thm]{Fact}
\newtheorem{lem}[thm]{Lemma}
\newtheorem{prb}[thm]{Problem}
\newtheorem{prp}[thm]{Proposition}
\newtheorem{qst}[thm]{Question}
\newtheorem{rem}[thm]{Remark}

\def\G{{\Gamma}}

\def\f{{\phi}}

\def\i{{\iota}}

\def\ps{{\psi}}

\def\Sg{{\Sigma}}

\def\V{{\Lambda}}

\def\cA{{\cal A}}
\def\cB{{\cal B}}
\def\cC{{\cal C}}
\def\cD{{\cal D}}

\def\cG{{\cal G}}

\def\cP{{\cal P}}
\def\cQ{{\cal Q}}

\def\sN{{\sf N}}
\def\sP{{\sf P}}

\def\cdg{{\textsf{\textbf{CD}}}}
\def\gdg{{\textsf{\textbf{SD}}}}
\def\pdg{{\textsf{\textbf{PD}}}}

\def\zD{{\mathbb D}}

\def\zZ{{\mathbb Z}}

\def\bG{{\mathbf{G}}}
\def\bH{{\mathbf{H}}}

\def\bx{{\mathbf{x}}}
\def\by{{\mathbf{y}}}

\def\mfC{{\mathfrak{C}}}

\def\dist{{\sf dist}}

\def\len{{\sf length}}

\def\po{{\sf P1}}
\def\pt{{\sf P2}}

\def\3sat{{\sf 3{\small SAT}}}

\def\odd{{\sf odd}}
\def\one{{\sf one}}
\def\four{{\sf four}}
\def\nx{{\sf Next}}
\def\pv{{\sf Prev}}

\def\dth{{d^{\rm th}}}

\def\mt{{\emptyset}}
\def\qed{{\hfill $\mathbin{\gbox}$\bigskip}}

\def\rar{{\rightarrow}}

\def\sqr#1#2{{\vcenter{\hrule height.#2pt
        \hbox{\vrule width.#2pt height#1pt \kern#1pt
                \vrule width.#2pt}
        \hrule height.#2pt}}}
\def\gbox{{\mathchoice\sqr45\sqr45\sqr{2.1}3\sqr{1.5}3}}

\definecolor{brwn}{RGB}{140, 70, 20}
\definecolor{gren}{RGB}{  0,140, 10}

\definecolor{primo}{RGB}{128,0,128}
\definecolor{segundo}{RGB}{255,92,0}
\definecolor{primox}{RGB}{128,0,128}
\definecolor{segundox}{RGB}{255,92,0}

\definecolor{fir}{RGB}{0,104,180}
\definecolor{sec}{RGB}{200, 0, 50}
\definecolor{thi}{RGB}{70, 100, 30}
\definecolor{fou}{RGB}{140, 70, 20}
\definecolor{fif}{RGB}{100, 0, 250}
\definecolor{darkblue}{RGB}{0,0,250}
\definecolor{mnt}{RGB}{167,230,215}

\newcommand{\fs}[1]{\textcolor{blue}{\sf{#1}}}
\newcommand{\gh}[1]{\textcolor{gren}{\sf{#1}}}
\newcommand{\me}[1]{\textcolor{brwn}{\sf{#1}}}
\newcommand{\tp}[1]{\textcolor{orange}{\sf{#1}}}
\newcommand{\up}[1]{\textcolor{red}{\sf{#1}}}

\begin{document}

\title{A New Dominating Set Game on Graphs}

\author{
Sean Fiscus 
    \footnotemark[1]
\and
Glenn Hurlbert 
    \footnotemark[2]
\and
Eric Myzelev
    \footnotemark[3]
\and
Travis Pence
    \footnotemark[4]
}

\date{}

\maketitle

\footnotetext[1]{
Department of Mathematics,
Duke University, 
Durham, North Carolina, USA,
\texttt{sean.fiscus@duke.edu}
}

\footnotetext[2]{
Department of Mathematics and Applied Mathematics, 
Virginia Commonwealth University, 
Richmond, Virginia, USA,
\texttt{ghurlbert@vcu.edu}
}

\footnotetext[3]{
Department of Mathematics, 
University of Pennsylvania,
Philadelphia, Pennsylvania, USA,
\texttt{myzelev@sas.upenn.edu}
}

\footnotetext[4]{
Department of Computer Science, 
University of Wisconsin,
Madison, Wisconsin, USA,
\texttt{tnpence@wisc.edu}
}

\begin{abstract}
We introduce a new two-player game on graphs, in which players alternate choosing vertices until the set of chosen vertices forms a dominating set.
The last player to choose a vertex is the winner.
The game fits into the scheme of several other known games on graphs.
We characterize the paths and cycles for which the first player has the winning strategy.
We also create tools for combining graphs in various ways (via graph powers, Cartesian products, graph joins, and other methods) for building a variety of graphs whose games are won by the second player, including cubes, multidimensional grids with an odd number of vertices, most multidimensional toroidal grids, various trees such as specialized caterpillars, the Petersen graph, and others. 
Finally, we extend the game to groups and show that the second player wins the game on abelian groups of even order with canonical generating set, among others.
\end{abstract}

\medskip


\newpage

\section{Introduction}
\label{s:Intro}

In this work we introduce a new combinatorial game, which we call the {\it snooker domination} (\gdg) {\it game}.
On a given graph, two players take turns choosing vertices until the set of chosen vertices forms a dominating set. 
The winner is the player who was last to choose a vertex. 
This game is similar to one introduced by Bre{\v s}ar, Klav{\v z}ar, and Rall \cite{Bre_ar_2013}, called the {\it domination game}, which we refer to as the {\it classical} version.
The main difference is that, in the classical game, each move must dominate a vertex not currently dominated; that is, we allow moves that choose a dominated vertex whose neighbors are also dominated (we think of this as {\it snookering}).
This simple change leads to completely different game behavior by preventing the normal machinery of game sums.
See \cite{BrMaMaSa} for recent results on normal play in the classical domination game.

We remark that this game generalizes the following natural game on groups, for which we derive several results in this paper. 
For a group $\G$ with generating set $\Sg$, where $\Sg$ is closed under inverses, players alternate choosing elements of $\G$ until every unchosen element $h$ equals $h'g$ for some chosen $h'\in\G$ and some $g\in\Sg$.
In this instance, the game played on a group is equivalent to the game played on the Cayley graph of that group.

All graphs $G=(V,E)$ are assumed to be finite and connected. 
We denote the path on $n$ vertices by $P_n$ and the cycle on $n$ vertices by $C_n$, and refer to the degree 1 vertices of a path as endpoints, and the degree 2 vertices as interior vertices.
The open (closed) neighborhoods of a vertex $u$ and of a set of vertices $U$ are denoted $N(u)$ (resp. $N[u]=N(u)\cup\{u\}$) and $N(U)=\cup_{u\in U}N(u)$, respectively.
We also define the set notation $N_X(u)=N(u)\cap X$.
For $u, v\in V(G)$, the function $\dist_G(u,v)$ measures the distance between them; i.e., the number of edges in a shortest $uv$-path.
This allows us to define the {\it closed distance-$d$ neighborhood} of $u$: $N^d[u]=\{v\in V\mid\dist(u,v)\le d\}$.
A set of vertices $D\subseteq V$ is called a {\it dominating set} if $V=N[D]$.
More generally, $D$ is a {\it distance-$d$ dominating set} if $V=N^d[D]$.

\subsection{History}
\label{ss:Hist}

There is a large body of work in graph theory that studies, not the graphs themselves, but the activities or games involving the placement of objects on the vertices of a graph and movements of them along edges under certain rules for various purposes.
Some versions of these model the spread of information or disease, supply-demand optimization, network search, and other applications, while others are studied purely for game-theoretic research.
Autonoma such as Conway's Game of Life \cite{ConwayLife} and bootstrap percolation \cite{BollRior} can be thought of as 0-person games, while others such as chip firing \cite{BjoLovSho}, zero forcing \cite{AIM}, and graph burning \cite{Bonato} can be thought of as 1-person games, or puzzles.
The objects of some 2-person games like cops \& robbers \cite{CopsRobbers} and graph pebbling \cite{Chung,HurlKent} are moved around by players, while some like Go \cite{Berlekamp} and Hex \cite{Hayward} are not.
Our game is of this last category.

Several games like ours have been studied.
For example, one can play a game in which two players alternate choosing vertices so that the set of chosen vertices is independent.
Phillips and Slater \cite{PhilSlatEnclave} initiated the study of the case in which the last player to chose a vertex wins, while Phillips and Slater \cite{PhilSlat} and Goddard and Henning \cite{GoddHenn} have studied the case of one player trying to maximize and the other player trying to minimize the size of the final chosen set.
Huggan, Huntemann, and Stevens \cite{HugHunSte} studied the game in which players alternate choosing vertices of a hypergraph $H$ (in their case a block design) so that the chosen set does not contain any edge of $H$, with the winner being last player able to choose a vertex.

As mentioned, in the game we study here, the players alternate choosing vertices until the chosen vertices form a dominating set.
In this context, 
Bre{\v s}ar, Henning, Klav{\v z}ar, and Rall \cite{BrHeKlRa} and Bre{\v s}ar, Klav{\v z}ar, and Rall \cite{BreKlaRal} studied the ``maximizer-minimizer'' version of the game, while Henning, Klav{\v z}ar, and Rall \cite{HenKlaRal} did the same for total dominating sets.
(Another interesting maximize-minimize variation is studied by Alon, Balogh, Bollob\'as, and Szab\'o \cite{AlBaBoSz}, in which players alternate orienting the edges of a graph until the entire graph is oriented, observing the size of the resulting oriented domination number.)

Here we study the winner-loser version, the winner being the player who first creates a dominating set from the chosen vertices.

\subsection{Game Theory Basics}
\label{ss:GameTheory}

We follow the traditional language and notation of \cite{BerConGuy,Fraenkel}.
Because a \gdg\ game is a finite, perfect information game with no ties, we know that one of the players has a winning strategy.
At every stage of the game, the player who is next to play is denoted \nx\, with the other player denoted \pv\ (thought of as ``previous'').
At the beginning of the game, we say that Player 1 is next, and write $\po=\nx$, with Player 2 being $\pt=\pv$.
For clarity, notice that after the first move we have $\nx=\pt$ and $\pv=\po$;
That is, \po\ and \pt\ are the permanent names of the players, while \nx\ and \pv\ describe which of them is about to play.
Then the set \sN\ is defined to be the set of all positions of the game for which \nx\ has the winning strategy, while \sP\ is defined to be the set of all positions of the game for which \pv\ has the winning strategy.
By identifying a game with its initial position, we can say that the game is in \sN\ (or \sP) if its initial position is in \sN\ (or \sP).

Every position $\cQ$ corresponds to a set $B$ of chosen vertices, so we define the {\it size} of $\cQ$ to be $|B|$.
Because \nx\ must always choose a vertex that has not yet been chosen, we will abuse notation slightly by writing $u\in\cQ$ to mean that $u$ is a move available for \sN\ to play from position $\cQ$.
The position created from $\cQ$ by \nx\ choosing the vertex $u\in\cQ$ is denoted $\cQ(u)$, and corresponds to the chosen set $B\cup\{u\}$.
Thus, \sP\ and \sN\ can be calculated recursively by first placing the position corresponding to $V(G)$ in \sP.
Then we repeatedly consider all positions $\cQ$ of size one less than prior: if $\cQ$ corresponds to a dominating set then place it in \sP; otherwise, if $\cQ(u)\in\sP$ for some $u$ then place it in \sN; otherwise, place it in \sP.
A common technique in game theory for proving that a position is in \sP\ is to display a {\it pairing strategy}: a function $\f$ from the possible moves for \sN\ to those for \sP, along with a condition $C$ such that all winning positions satisfy $C$ and, for every position $\cQ$ that satisfies $C$ and every move $u$ from $\cQ$, $\cQ(u)$ does not satisfy condition $C$ and $\cQ(u,\f(u))$ does satisfy $C$.
We make use of this idea below.

Let \gdg\ to refer to domination games on general graphs; for games on paths and cycles, we refer to them by \pdg\ and \cdg, respectively.

\subsection{\gdg\ and \pdg\ Game Notations}
\label{ss:Games}

To assist with the analysis of game positions and player strategies, it will be useful to keep track of more than just the chosen vertices $B$; in particular, we will also distinguish vertices that are or are not dominated by $B$.
When the graph $G$ is understood (in our case $G=P_n$ or $C_n$), we will leave it out of the positional notation. 
With this finer perspective, a \gdg\ {\it position} $\cQ$ on $G$ is a partition $\{B,S,W\}$ of $V(G)$ (for {\it black}, {\it shaded}, and {\it white}) such that no white vertex is adjacent to a black vertex.
The interpretation here is that $B$ is the set of chosen vertices, $S=N(B)-B$ is the set of non-chosen vertices dominated by $B$, and $B$ is a dominating set of $G$ if and only if $W=\mt$.
\gdg\ positions with $W=\mt$ are called {\it trivial}; all other games are {\it nontrivial}.
Thus all trivial positions are in \sP.
We will use the notations $B_\cQ$, (resp. $S_\cQ$, $W_\cQ$) when necessary to indicate the black (resp. shaded, white) vertices of position $\cQ$.

Define two \gdg\ positions $\cQ_1=\{B_1,S_1,W_1\}$ and $\cQ_2=\{B_2,S_2,W_2\}$ on graphs $G_1$ and $G_2$ to be {\it isomorphic}, written $\cQ_1\cong\cQ_2$, if there is a graph isomophism $\varphi:G_1-B_1\to G_2-B_2$ such that $\varphi(S_1)=S_2$ and $\varphi(W_1)=W_2$. 
Clearly, if $\cQ_1\cong\cQ_2$ then $\cQ_1\in \sP$ if and only if $\cQ_2\in \sP$.

Now define the game $\zD=\zD(G)$ with initial position $\cQ_0$ having $W=V(G)$.
At any stage of the game with position $\cQ$,
if $W_\cQ\not=\mt$ then \nx\ chooses some vertex $u\in S_\cQ\cup W_\cQ$, resulting in the position $\cQ(u)$ having $B_{\cQ(u)}=B_\cQ\cup\{u\}$, $S_{\cQ(u)}=(S_\cQ-\{u\})\cup N_W(u)$, and $W_{\cQ(u)}=W_\cQ-(\{u\}\cup N(u))$.

\subsection{Involutions}
\label{ss:Involutions}

For a graph $G$, an {\it involution} of $G$ is an automorphism of $G$ of order two.
We say that an involution $\f$ is 
$d$-{\it involutionary} if, for all $v\in V(G)$, $\dist_G(v,\f(v))\ge d$. 
Define a graph to be {\it $d$-involutionary} if it has $d$-involution.
Observe that an involution $\f$ is 1-involutionary if and only if it has no fixed points, and is 3-involutionary if and only if $N[v]\cap N[\f(v)]=\mt$ for all $v\in V(G)$.
The definition can be extended to a position $\cQ=\{B,S,W\}$ on a graph $G$ by requiring that $\f$ be an automorphism of $G-B$; that is, once a vertex is chosen by a player, it is removed from consideration.
The importance of a 3-involution is that it gives rise to a pairing strategy, as we show in Lemma \ref{l:Involution}.

Suppose that $\G$ is a group with elements $\V$ and generating set $\Sg$ (which we consider as closed under inverses). 
We will always assume that $\Sg$ is closed under inverses. 
For the purposes of this paper, it is not important what set of relations $\G$ has; consequently, we will use the notation $\G=(\V,\Sg)$ to denote any such group.
The Cayley graph $\mfC(\G,\Sg)$ is defined to be the undirected graph $G=(V,E)$ with $V=\V$ and $E=\{\{g,h\}\in\binom{\V}{2}\mid gh^{-1}\in\Sg\}$.
We say that $(\G,\Sg)$ is {\it involutionary} if there is some $g\in\V-\Sg$
of order 2, and {\it 3-involutionary} if $g$ also cannot be expressed as the product of two generators.
Thus, if $(\G,\Sg)$ is $3$-involutionary then so is $\mfC(\G,\Sg)$.
Additionally, we define the game $\zD(\G,\Sg)$ to be the game described above, in which players alternate choosing elements of $\G$ until every unchosen element $h$ equals $h'g$ for some chosen $h'$ and some $g\in\Sg$.
The last person able to choose an element is the winner.
It is easy to see that $\zD(\G,\Sg)=\zD(G)$, where $G=\mfC(\G,\Sg)$.

\subsection{Graph constructions}
\label{ss:GraphCon}

Among our results are statements involving the following graph constructions.\label{ref2:overlap}

\bigskip
\noindent{\bf Graph power.}
Given a graph $G=(V,E)$ and positive integer $d$, the {\it $\dth$ power of $G$} is denoted $G^{(d)}=(V,F)$, where $F=\{\{x,y\}\in\binom{V}{2}\mid \dist_G(x,y)\le d\}$.
Observe that $N_G^d[u]=N_{G^{(d)}}[u]$.
For example, if $M$ is a perfect matching in $K_{2t}$, then $C_{2k}^{(k-1)}\cong K_{2k}-M$.

\bigskip
\noindent{\bf Cartesian product.}
Given graphs $G_1=(V_1,E_1)$ and $G_2=(V_2,E_2)$, the {\it Cartesian product of $G_1$ and $G_2$} is denoted $G_1\mathbin{\gbox} G_2=(V,E)$, where $V=V_1\times V_2$ and $E= 
\{\{(u,v_1),(u,v_2)\}\mid \{v_1,v_2\}\in E_2\}\cup
\{\{(u_1,v),(u_2,v)\}\mid \{u_1,u_2\}\in E_1\}$.
For example, the 2-dimensional $m\times k$ grid on $mk$ vertices is isomorphic to $P_m\mathbin{\gbox} P_k$.
Also, the $d$-dimensional cube equals $P_2^d = P_2\mathbin{\gbox} P_2\mathbin{\gbox}\cdots \mathbin{\gbox} P_2$ ($d$ copies of $P_2$).

\bigskip
\noindent{\bf Join.}
Given graphs $G=(V,E)$ and $H=(W,F)$, each with distinguished vertices $x\in V$ and $y\in F$, the {\it join of $(G,x)$ and $(H,y)$} is denoted $(G,x)\cdot (H,y) = (V',E')$, where $V'=V\cup W$, setting $x=y$, and $E'=E\cup F$.
For example, for any such $x$ and $y$, the join $(C_n,x)\cdot (C_m,y)$ can be drawn to resemble a figure eight or infinity symbol.

\bigskip
\noindent
{\bf Dangling.}
\label{ref1:dangling}
Given a graph $G=(V,E)$ with a sequence of distinct, distinguished vertices $\bx=(x_1,\ldots,x_k)$ and a sequence of graphs $\bH=(H_1,\ldots,H_k)$ (each $H_i=(W_i,F_i)$) with corresponding sequence of distinguished vertices $\by=(y_1,\ldots,y_k)$ (each $y_i\in W_i$), a {\it dangling of $\bH$ on $G$} is denoted $(G,\bx)\cdot (\bH,\by)=(V',E')$, where $V'=V\cup\left(\cup_{i=1}^k W_i\right)$, setting each $x_i=y_i$, and $E'=E\cup\left(\cup_{i=1}^k F_i\right)$.
One can think of dangling as a sequence of joins done in parallel.
For example, if $\bx$ is any ordering of $V$ and each $H_i\cong K_2$ with any choice of $y_i$, then $(C_n,\bx)\cdot (\bH,\by)$ 
is known as the {\it $k$-sunlet} graph $S_k$, shown in Figure \ref{fig:sunlet}, below.

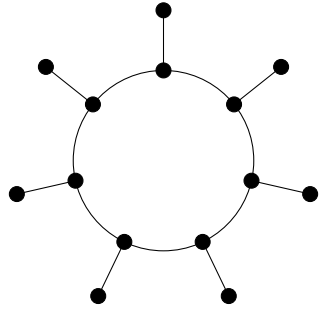
\begin{figure}[h]
\centering
\begin{tikzpicture}[scale=0.4]
\tikzstyle{every node}=[draw,fill=black,shape=circle,inner sep=2pt];
\def \inner {3cm}
\def \outer {5cm}
\def \ang {360/7}
\draw (0,0) circle (\inner);
\foreach \i in {0,...,6}{
    \node at ({90+(\i*\ang)}:\inner) {};
    \node at ({90+(\i*\ang)}:\outer) {};
    \draw ({90+(\i*\ang)}:\inner) -- ({90+(\i*\ang)}:\outer);
    }
\end{tikzpicture}
\caption{The 7-sunlet graph $S_7$.}
\label{fig:sunlet}
\end{figure}

\bigskip
\noindent
{\bf Bridging.}
\label{ref1:bridging}
Given a graph $G=(V,E)$ with a distinguished vertex $x$ and a sequence of graphs $\bH=(H_1,\ldots,H_k)$ (each $H_i=(W_i,F_i)$) with corresponding sequence of distinguished vertices $\by=(y_1,\ldots,y_k)$ (each $y_i\in W_i$), a {\it bridging of $(G,x)$ and $(\bH,\by)$} is denoted $(G,x)$--$ (\bH,\by)=(V',E')$, where $V'=V\cup\left(\cup_{i=1}^k W_i\right)$ and $E'=E\cup\left(\cup_{i=1}^k (F_i\cup\{\{x,y_i\}\})\right)$.
(The emdash notation is meant to suggest connecting two graphs by an edge, in contrast to the $\cdot$ notation of dangling, which connects two graphs on a common vertex.)
If each $(H_i,y_i)=(H,y)$ we denote this operation by $(G,x)$--$(H,y)^k$, with the exponent suppressed when $k=1$.
For example, for degree 2 vertices $x$ and $y$, the bridging $(P_3,x)$--$ (P_3,y)$ yields the internally 3-regular caterpillar shown in Figure \ref{fig:cater}.

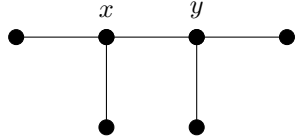
\begin{figure}[h]
\centering
\begin{tikzpicture}[scale=0.4]
\tikzstyle{every node}=[draw,fill=black,shape=circle,inner sep=2pt];
\def \len {3cm}
\foreach \i in {0,...,3}
    \node at ({\i*\len},{1*\len}) {};
\draw (0,{1*\len}) -- (3*\len,{1*\len});
\foreach \i in {1,2}{
    \node at ({\i*\len},0) {};
    \draw ({\i*\len},0) -- ({\i*\len},{1*\len});
    }
\node[label=$x$] at ({1*\len},{1*\len}) {};
\node[label=$y$] at ({2*\len},{1*\len}) {};
\end{tikzpicture}
\caption{The internally 3-regular caterpillar $(P_3,x)$--$ (P_3,y)$.}
\label{fig:cater}
\end{figure}

\subsection{Our results}
\label{ss:Results}

The following lemma is fundamental to carving out many results.

\begin{lem}[Involution Lemma]
\label{l:Involution}
If $\cQ$ is a 3-involutionary position, then $\cQ\in \sP$. 
In particular, if the graph $G$ is 3-involutionary then $\zD(G)\in\sP$.
\end{lem}

\begin{proof}
We show that $\cQ\not\in\sN$ by proving that if $\cQ$ is a 3-involutionary position then, for all $u\in\cQ$ we have (1) $\cQ(u)$ is not trivial and (2) $\cQ(u,v)$ is 3-involutionary for some $v\in\cQ(u)$.
Thus \po\ can never end the game.

Indeed, let $\f$ be 3-involution of $\cQ$, suppose that \nx\ chooses $u\in\cQ$, and define $v=\f(u)$.
Because $\f$ is an involution, it partitions the vertices into pairs $\{z,\f(z)\}$ such that the elements of each pair are of the same ``color'' (white or shaded).
Thus, if $u\in W_\cQ$ then $v\in W_\cQ$.
Because $\f$ is 3-involutionary, $v\in W_{\cQ(u)}$.
Also, if $u\in S_\cQ$ then $v\in S_\cQ$.
So, if $\cQ(u)$ is trivial, then there is some $x\in N_{W_\cQ}(u)$, which implies that $\f(x)\in N_{W_\cQ}(v)$.
But because $\f$ is 3-involutionary, $\f(x)\not\in N[u]$, and so $f(x)\in W_{\cQ(u)}$, a contradiction.
This proves part (1).
For part (2), we observe that $\f$ is a 3-involution of $\cQ(u,v)$.
\end{proof}

The Involution Lemma yields several immediate results.

\begin{cor}
\label{c:Cubes}
For all $d\ge 2$ we have $\zD(P_2^d)\in\sP$, and $\zD(P_2)\in\sN$.
\end{cor}

\begin{proof}
Label the two vertices of $P_2$ by 0 and 1; then this labeling naturally extends to the vertices of $G=P_2^d$ by labels in $\{0,1\}^d$.
Define the {\it antipodal} mapping $\f$ on $G$ by $\f(x_1,\ldots,x_d)=(1-x_1,\ldots,1-x_d)$.
For $d=1$, the statement is obvious.
For $d=2$, if \po\ plays $u$ then \pt\ wins by playing $\f(u)$.
For $d\ge 3$, $\f$ is 3-involutionary, so $\zD(G)\in\sP$ by Lemma \ref{l:Involution}.
\end{proof}

\begin{cor}
\label{c:InvoGroup}
For every 3-involutionary group $\G$ with generating set $\Sg$, we have $\zD(\G,\Sg)\in\sP$.
\end{cor}

\begin{proof}
Let $(\G,\Sg)$ be a 3-involutionary group, and $g$ be an order 2 element of $\V-\Sg$ that cannot be expressed as the product of two generators. 
Define $G=\mfC(\G,\Sg)=(V,E)$, and let $\f:G\to G$ be an automorphism given by $\f(v)=gv$.
Then $\f$ is also an involution on $G$. 
To see that $\f$ is 3-involutionary, first notice that $\f(u)=gu\notin N[u]$ since $g\notin \Sigma$. 
Then supposing, for contradiction, that there is an element $v\in N(u)\cap N[\f(u)]$, we see that there are generators $a,b\in \Sg$ such that $abu=\f(u)=gu$.
Canceling $u$ yields $ab=g$, contradicting the fact that $g$ cannot be written as a product of at most two generators. 
This shows that $\varphi$ is 3-involutionary.
\end{proof}

\begin{exa}
\label{e:CyclePower}
Let $n\ge 6$ be even and let $k<n/4$ be a positive integer. 
Consider the game in which two players alternate choosing numbers from $\{0,1,2,\ldots,n-1\}$ until every unchosen number is at distance at most $k$ from some chosen number, where the distance between two numbers $a\le b$ equals $\min\{b-a,n+a-b\}$, and the winner is the last player to move.
This game can be modeled by the game on the group $\G=\zZ_n$, with generating set $\Sg=\{\pm 1, ...,\pm k\}$, where $\zZ_n$ denotes the cyclic group of order $n$. 
Then $n/2\notin \Sg$ and cannot be expressed as the sum of two elements in the generating set, which implies that the automorphism $\f(x)=x+n/2\mod n$ is 3-involutionary, and so, by Corollary \ref{c:InvoGroup}, $\zD(\G,\Sg)\in\sP$.
We can also observe that $\mfC(\G, \Sg)\cong C_n^{(k)}$, and thus this also shows that $C_n^{(k)}\in \sP$.
\end{exa}

\begin{exa}
\label{e:CycleProduct}
Let $n_1,...,n_d$ be positive integers, each at least 3,
and suppose that $n_1$ is even and greater than or equal to $6$.
Consider the group $\G=
\zZ_{n_1}\times\cdots\times \zZ_{n_d}$.
Let $e_j\in \G$ be the element with 1 in coordinate j and 0 in every other coordinate, and let the {\em canonical} generating set of $\G$ be $\Sg=\{\pm e_j\:|\: 1\le j\le d\}$.
Then, since $(n_1/2,0,...,0)\notin \Sg$ and cannot be expressed as $e_j+e_\ell$ for any $1\le j,\ell\le d$, by Corollary \ref{c:InvoGroup} we conclude that $\zD(\G, \Sg)\in \sP$.
Furthermore, we can observe that $\mfC(\G, \Sg)\cong C_{n_1}\mathbin{\gbox}\cdots\mathbin{\gbox} C_{n_d}$, which shows that $C_{n_1}\mathbin{\gbox}\cdots\mathbin{\gbox} C_{n_d}\in\sP$.
\end{exa}

For an abelian group $\G$, we say that the representation $\G\cong\zZ_{n_1}\times\cdots\times\zZ_{n_d}$ is {\it good} if $n_1\ge 6$ or if $\G\cong \zZ_2^k\times\zZ_4^m$ for some $k$ and $m$.
Notice that every even-order abelian group has a good representation.
Then Example \ref{e:CycleProduct} proves the following proposition unless $\G\cong \zZ_2^k\times\zZ_4^m$, for which case Corollary \ref{c:Cubes} applies because $\mfC(\zZ_2^k\times\zZ_4^m, \Sigma)=P_2^{k+2m}$, where $\Sigma$ is the canonical generating set.

\begin{prp}
\label{p:Abelian}
Let $\G\not\cong\zZ_2$ be an even-order abelian group with good representation and corresponding canonical generating set $\Sg$.
Then $\zD(\G,\Sg)\in\sP$.
\hfill$\Box$
\end{prp}

We will defer the proof of the following Theorem to Section \ref{ss:InvoP}.

\begin{thm}
\label{t:InvoP}
The game $\zD(G)\in\sP$ for the following graphs $G$:
\begin{enumerate}
    \item 
    \label{cp:Cpower}
    $C_n^{(k)}$, for even $n\ge 6$ and $1\le k<n/4$.
    \item 
    \label{cp:CartProd}
    $G_0\mathbin{\gbox} H$, whenever (a) $G_0$ is 3-involutionary and $H$ is any graph, or (b) $G_0$ is 2-involutionary and $H$ is 1-involutionary.
    \item 
    \label{cp:Dangle2}
    $(G_0,(x,\f(x))\cdot ((H,H),(y,y))$, such that $\f$ is a 3-involution
    of $G_0$, $H=(W,F)$ is any graph, and $y\in W$.
    \item 
    \label{cp:DangleEven}
    $(G_0,\bx)\cdot (\bH,\by)$, $|V(G_0)|\ge 2$, $\bx=V(G_0)$, and each $H_i\cong K_{1,m_i}$ with odd $m_i\ge 1$ and $y_i$ is a dominating vertex of $H_i$.
    \item 
    \label{cp:Bridge}
    $(G_0,x)$--$(G_0,x)^k$, such that $x$ is a cut-vertex of $G_0$, some component of $G_0-x$ is a singleton, and $k$ is odd.
\end{enumerate}
\end{thm}

Paths and odd cycles are examples of graphs that are not
3-involutionary,
so they require additional techniques.
We are able to characterize which path and cycle games are won by Player 1, as follows.

\begin{thm}
\label{t:DomGameCycle}
For $n\ge 3$, $\zD(C_n)\in\sN$ if and only if $n$ is odd but not equal to 5.
\end{thm}

The analysis of $\cdg$ games follows from that of $\pdg$ games, which is given by the following theorem.

\begin{thm}
\label{t:DomGamePath}
For $n\ge 2$, $\zD(P_n)\in\sN$ if and only if $n$ is odd or $n\in\{2, 6, 8, 10, 12\}$.
\end{thm}

In Section \ref{s:Lemmas} we develop the lemmas necessary to prove Theorems \ref{t:InvoP}--\ref{t:DomGamePath}, which we do in Section \ref{s:Proofs}.
Section \ref{s:Specifics} contains theorems about games on specific graphs such as caterpillars, Cartesian products of paths and cycles, and the Petersen graph.
We offer several interesting questions, open problems, and a conjecture in Section \ref{s:Comments}.

\section{Key Lemmas}
\label{s:Lemmas}

\subsection{\gdg\ games}
\label{ss:gdg}

For a set of \gdg\ positions $\cQ_i$ on corresponding graphs $G_i$ ($1\le i\le k$), define the {\it game sum} $\cQ=\oplus_{i=1}^k \cQ_i$ to have $B_\cQ=\cup_{i=1}^k B_{\cQ_i}$, $S_\cQ=\cup_{i=1}^k S_{\cQ_i}$, $W_\cQ=\cup_{i=1}^k W_{\cQ_i}$ on the disjoint graph union $+_{i=1}^k G_i$ so that, for each $i\not=j$, no move in $\cQ_i$ affects $\cQ_j$.
Each $\cQ_i$ is called a {\it component} of $\cQ$.
While game sums are normally fairly simple to interpret, the complexity in this case is that players can continue to play on the shaded vertices of a component that has been ``won'' (is trivial).
Again, this is the crucial difference between our game and the classical domination game.

\begin{fct}
\label{f:parity}
For any position $\cQ$ and any vertex $u\in \cQ$, we have $|S_{\cQ(u)}|+|W_{\cQ(u)}|=|S_\cQ|+|W_\cQ|-1$.
\end{fct}

\begin{proof}
This holds because $|B_{\cQ(u)}|=|B_\cQ|+1$.
\end{proof}

\begin{lem}
\label{l:Iso}
For \gdg\ games $\cA$ and $\cB$, if $\cA\cong\cB$ then $\cA\oplus\cB\in\sP$.
\end{lem}

\begin{proof}

Let $\cA$ and $\cB$ be \gdg\ games, $f:\cA\rar\cB$ be an isomorphism. 
Define $\f:\cA\oplus\cB\to \cA\oplus \cB$ by setting $\f(a)=f(a)$ for all $a\in \cA$ and $\f(b)=f^{-1}(b)$ for all $b\in\cB$. 
Then $\f$ is 3-involution of $\cA\oplus\cB$ and so, by Lemma \ref{l:Involution}, $\cA\oplus\cB\in \sP$.
\end{proof}

\subsection{\pdg\ games}
\label{ss:pdg}

To describe \pdg\ games, we first make an observation that will yield simpler notation that will facilitate our analysis.
The observation is that every interior vertex of a path is a cut vertex, so when \nx\ chooses an interior vertex $u$, the game on the original path $P$ becomes a sum of games on two paths defined by $P-u$.
Thus we define the position $\cP_k^i$ to be the path on $k$ vertices having $i$ shaded endpoints.
Then, for example, if $u$ has distance at least two from each endpoint of $\cP_k^i$ then $\cP_k^i(u)\cong\cP_a^h\oplus\cP_b^j$, for some $a,b\ge 2$, $a+b=k-1$, $1\le h\le 2$, $1\le j\le 2$, and $h+j=2+i$.
When $u$ is an endpoint or neighbor of one, it is slightly trickier to write a general formula, but simple to calculate a particular instance; for example, if $u$ is a neighbor of the shaded endpoint of $\cP_7^1$ then $\cP_7^1(u)\cong\cP_1^1\oplus\cP_5^1$.
Notice that, with this notation, there is no mention of $B$ since it no longer exists --- the chosen vertices are not colored black but instead are removed --- so the original definition of $S$ no longer applies.
Instead we know now that $S$ is a subset of the endpoints of paths, and endpoints next to a chosen vertex become shaded.

We define a \pdg\ position $\cQ$ to be a finite sum of such $\cP_k^i$ positions.
We say a \pdg\ position or game $\cQ=\oplus_{j=1}^m\cP_{k_j}^{i_j}$ is {\it standard} if $i_j=2$ when $k_j>1$, and $i_j=1$ otherwise; that is, every endpoint is shaded. 
Thus, if $\cQ$ is standard then, for every move $u$ in $\cQ$, we have that $\cQ(u)$ is standard.
To simplify the notation of standard positions for most of this paper, we write $\cP_k$ in place of $\cP_k^2$ for $k>1$ and $\cP_1$ in place of $\cP_1^1$, with $\cP_0$ denoting the empty position (having no vertices).

Given standard $\cQ = \oplus_{i=1}^m \cQ_i$ where $Q_i\cong\cP_{n_i}$, we define the functions $\one(\cQ)=|\{i\mid n_i=1\}|$, $\four(\cQ)=|\{i\mid n_i=4\}|$, and $\odd(\cQ)=|\{i\mid n_i {\rm\ is\ odd}\}|$.
We say that $\cQ$ is {\it even} if $|W_\cQ|+|S_\cQ|$ is even.
Equivalently, $\odd(\cQ)$ is even.
We say that $\cQ$ is {\it totally even} if both $|W_\cQ|$ and $|S_\cQ|$ are even and $\four(\cQ)$ is even.
Note that $|S_\cQ|$ is even if and only if $\one(\cQ)$ is even.
Thus, an even $\cQ$ is totally even if and only if $\one(\cQ)$ and $\four(\cQ)$ are both even.

Let $\f$ denote the automorphism of any path that swaps its endpoints.
We say that $x_i$ is the {\it center} of $\cQ_i$ if $\f(x_i)=x_i$.
Note that if $x_i$ is the center of $\cQ_i$ then $n_i$ is odd.

\begin{lem}
\label{l:totally_even}
Let $\cQ = \oplus_{i=1}^m \cP_{n_i}$ be standard and nontrivial such that $|S_\cQ|+|W_\cQ|$ is odd. Then \nx\ has a move $u\in \cQ$ such that $\cQ(u)$ is totally even.
\end{lem}

\begin{proof}
By Fact \ref{f:parity}, $\cQ(u)$ is even for any move $u$, and we therefore recall that $\cQ(u)$ is totally even if and only if $\one(\cQ(u))$ and $\four(\cQ(u))$ are both even. 
We consider the following four cases:
\begin{quote}
\begin{itemize}
    \item $\one(\cQ)$ and $\four(\cQ)$ are both odd:\: Since $\four(\cQ)$ is odd we can write $\cQ\cong \cA\oplus \cP_4$ for some game $\cA$. Then \nx\ plays an interior vertex $u\in \cP_4$, so that $\one(\cQ(u))=\one(\cQ)+1$ and $\four(\cQ(u))=\four(\cQ)-1$.
    \item $\one(\cQ)$ is even and $\four(\cQ)$ is odd:\: Since $\four(\cQ)$ is odd we can again write $\cQ\cong \cA\oplus \cP_4$. Then \nx\ plays an endpoint $u\in \cP_4$, so that $\one(\cQ(u))=\one(\cQ)$ and $\four(\cQ(u))=\four(\cQ)-1$.
    \item $\one(\cQ)$ is odd and $\four(\cQ)$ is even:\: Since $\one(\cQ)$ is odd we can again write $\cQ\cong \cA\oplus \cP_1$. Then \nx\ plays the unique vertex $u\in \cP_1$, so that $\one(\cQ(u))=\one(\cQ)-1$ and $\four(\cQ(u))=\four(\cQ)$.
    \item $\one(\cQ)$ and $\four(\cQ)$ are both even:\: Since $\one(\cQ)$ is even, so is $|S_\cQ|$, which makes $|W_\cQ|$ odd. Thus there is some $i$ for which $n_i\ge3 $ is odd, and so we write $\cQ\cong \cA\oplus \cP_{n_i}$. Then \nx\ plays the center $u$ of $\cP_{n_i}$. By symmetry, $\one(\cQ(u))$ and $\four(\cQ(u))$ are both even.
\end{itemize}
\end{quote}
Therefore, in all cases there is a move $u$ such that $\cQ(u)$ is totally even.
\end{proof}

\begin{fct}
\label{f:nontrivial}
Let $\cQ = \oplus_{i=1}^m \cP_{n_i}$ be standard, nontrivial and totally even. 
Then for all vertices $u\in \cQ$, the position $\cQ(u)$ is nontrivial.
\end{fct}

\begin{proof}
The only even positions that can be won in one move are of the form $\left(\oplus_{i=1}^m \cP_1\right)\oplus\cP_4$, where $m$ is even. 
Such positions are not totally even. 
\end{proof}

\begin{lem}
\label{l:pairs}
For any \pdg\ positions $\cQ$, $\cA$, and $\cB$, with $\cA\equiv\cB$, we have that $\cQ\oplus\cA\oplus\cB\in\sP$ if and only if $\cQ\in\sP$.
\end{lem}

\begin{proof}
Let $\ps:\cA\rar\cB$ be an isomorphism.
If $\cQ\in\sP$ then whenever \po\ plays $u\in\cA$ (resp. $\cB$) then \pt\ plays $\ps(u)\in\cB$ (resp. $\ps^{-1}(u)\in\cA$).
Hence $\cQ\oplus\cA\oplus\cB\in\sP$.
On the other hand, if $\cQ\in\sN$, then \po\ plays $u\in\cQ$ such that $\cQ(u)\in\sP$.
By the above argument we have $\cQ(u)\oplus\cA\oplus\cB\in\sP$, and so $\cQ\oplus\cA\oplus\cB\in\sN$.
\end{proof}

\begin{thm}
\label{t:characterization}
Let $\cQ = \oplus_{i=1}^m \cP_{n_i}$ be standard and nontrivial. 
Then $\cQ\in\sP$ if and only if $\cQ$ is even and $|W_\cQ|\ge 4$, or $\cQ$ is isomorphic to 
one of the following, for any nonnegative $h$:
\begin{itemize}
    \item $\left(\oplus_{i=1}^m \cP_1\right)\oplus\left(\oplus_{i=1}^h \cP_2\right)\oplus\cP_3\oplus \cP_3$, where $m$ is even
    \item $\left(\oplus_{i=1}^m \cP_1\right)\oplus\left(\oplus_{i=1}^h \cP_2\right)\oplus\cP_4\oplus \cP_3$, where $m$ is odd
\end{itemize}
\end{thm}


We recall that $\cQ = \oplus_{i=1}^m \cP_{n_i}$ is a finite position and therefore must lead to a trivial position in finitely many steps. Thus, $\cQ$ is in exactly one of $\sP$ and $\sN$.

\begin{proof}[Proof of Theorem \ref{t:characterization}]
Say that $\cQ$ has {\it condition $C$} if $\cQ$ is totally even or $\cQ$ is even and $|W_\cQ|\ge 4$.
We argue that if $\cQ$ has condition $C$ then $\cQ(u)$ is nontrivial for all $u\in\cQ$, and then use Lemma \ref{l:totally_even} to conclude that $\cQ(u,v)$ has condition $C$.
Thus \po\ can never win $\cQ$; i.e. $\cQ\in\sP$.
Indeed, Fact \ref{f:nontrivial} takes care of the totally even case, and $|W_\cQ|\ge 4$ implies that $|W_{\cQ(u)}|\ge 1$ for all $u\in\cQ$. This shows that if $\cQ$ is even and $|W_\cQ|\ge 4$, then $\cQ\in \sP$.

Next, we look at the two cases listed in the statement of the theorem. We can assume that $h=0$ because, for any position $\cA$, $\cA\in N$ if and only if $\cA\oplus \cP_2\in N$.
First, suppose that $\cQ\cong \left(\oplus_{i=1}^m \cP_1\right)\oplus\cP_3\oplus \cP_3$, where $m$ is even. 
Then we can write it as $\cQ\cong \cA\oplus \cA$ where $\cA\cong \left(\oplus_{i=1}^{m/2} \cP_1\right)\oplus\cP_3$. 
Then $\cQ\in \sP$ by Lemma \ref{l:Iso}. 
Second, suppose that $\cQ\cong \left(\oplus_{i=1}^m \cP_1\right)\oplus\cP_4\oplus \cP_3$, where $m$ is odd. 
Without loss of generality, we can assume that $m=1$. 
If \po\ chooses $u\in \cP_3$ or $v\in W_{\cP_4}$, then \pt\ will choose $v$ or $u$, respectively, to win. 
If \po\ chooses $u\in S_{\cP_1}$ or $v\in S_{\cP_4}$, then \pt\ will choose $v$ or $u$, respectively, to yield $\cQ(u,v)\cong\cP_3\oplus\cP_3$, which \pt\ wins by Lemma \ref{l:Iso}.
Hence $\cQ\in\sP$.

Conversely, by contrapositive, we first suppose that $\cQ$ is odd. 
By Lemma \ref{l:totally_even}, \po\ has a move $u$ such that $\cQ(u)$ is totally even and thus in $\sP$. 
Therefore, $\cQ\in \sN$. 

The only remaining case is $\cQ$ is even, $|W_\cQ|<4$, and $\cQ$ is not isomorphic to one of the games listed in the statement of the theorem. If $\cQ\cong \left(\oplus_{i=1}^m \cP_1\right)\oplus\left(\oplus_{i=1}^h \cP_2\right)\oplus\cP_3\oplus \cP_3\oplus \cP_3$, then, $\pt$ plays the white vertex $u$ in one of the copies of $\cP_3$ and $\cQ(u)\in \sP$ by Lemma \ref{l:pairs}.
Otherwise, $\cQ\cong \left(\oplus_{i=1}^m \cP_1\right)\oplus\left(\oplus_{i=1}^h \cP_2\right)\oplus\cP_k$, where $k\in\{3,4,5\}$ and $m\equiv k\pmod{2}$.
For each $k\in \{3,4,5\}$, there are at most three white vertices and they are all consecutive. Thus, \po\ can win in one move.
\end{proof}
\begin{rem}
In the notation of Theorem \ref{t:characterization}, suppose that $|W_\cQ|\ge 10$. 
Then, after any move $u$ of \po, any response $v$ by \pt\ keeps the game even by Fact \ref{f:parity}. 
Also, $|W_{\cQ(u,v)}|\ge 4$ and so $\cQ(u,v)\in \sP$ by Theorem \ref{t:characterization}. 
Therefore, while there remain at least 10 white vertices, \pt's response to \po\ can be arbitrary.
\end{rem}

\begin{cor}
\label{c:path}
The position $\cP_k\in \sP$ if and only if $k$ is even and $k\ne 4$.
\end{cor}

\begin{proof}
    This follows as a special case of Theorem \ref{t:characterization}.
\end{proof}

\section{Proofs} 
\label{s:Proofs}

\subsection{Proof of Theorem \ref{t:InvoP}}
\label{ss:InvoP}

For the first 3 cases, we show that the graph under consideration 
has a 3-involution.
Then the results follow from Lemma \ref{l:Involution}.

\begin{enumerate}

\item 
This was proved in Example \ref{e:CyclePower}.
\hfill$\diamond$

\item
For case (a), let $\f':G_0\to G_0$ be a 3-involution, and define the involution $\f$ on $G_0\mathbin{\gbox} H$ by setting $\f(g, h)=(\f'(g),h)$.
For case (b), let $\f':G_0\to G_0$ be a 2-involution and $\f'':H\to H$ be a 1-involution, and define the involution $\f$ on $G_0\mathbin{\gbox} H$ by setting $\f(g, h)=(\f'(g),\f''(h))$.
In both cases, $\f$ is 3-involutionary.
\hfill$\diamond$

\item 
Let $G_0=(V, E)$ and $H=(W, F)$, and set $G=(G_0,(x,\f(x))\cdot ((H,H),(y,y))$. We will say that $W=\{w_1,...,w_n\}$, and define $W_i=\{w_1^i,...,w_n^i\}$ for $i=1,2$. Then we can write $V(G)=V\cup W_1\cup W_2$. Let $\f:G_0\to G_0$ be a 3-involution. Then define $\overline\f:G\to G$ by setting $\overline\f(u)=\f(u)$ for $u\in V(G)$, $\overline\f(w_j^1)=w_j^2$ for $w_j^1\in W_1$, and $\overline\f(w_j^2)=w_j^1$ for $w_j^2\in W_2$. Then $\overline\f$ is a 3-involution. 
\hfill$\diamond$

\item 
We prove this statement for a broader class of graphs and for a broader set of positions on them.
Let $\cG$ be the set of positions $\cQ$ on  graphs $G=((G_0,\bx)\cdot (\bH,\by))\cup G_1\cup G_2$, where $G_1\cup G_2$ is an independent set,
with the following properties: 
\begin{itemize}
    \item 
    $|V(G_0)|\ge 2$;
    \item 
    $\bx=V(G_0)$;
    \item 
    no $x_i\in \cQ_B$;
    \item 
    each $H_i\cong K_{1,m_i}$, with odd $m_i\ge 1$;
    \item 
    each $y_i$ is a dominating vertex of $H_i$;
    \item 
    $G_1$ is a set of evenly many isolated vertices in $\cQ_W$; and
    \item 
    $G_2$ is a set of evenly many isolated vertices in $\cQ_S$.
\end{itemize}
For such $\cQ\in\cG$, we prove by induction on $|\cQ|$ that $\cQ\in\sP$.
At each stage, when a player plays a vertex $v$ in some graph $G$, we remove black vertices so that $G(v)=G-v$.
By doing so, when isolated vertices occur, they get moved in $G_1\cup G_2$ according to their color.

Let $\cQ\in \cG$ and $k=|V(G_0)|$.
We write $N_{H_i}(y_i)=\{z_{i,1},\ldots,z_{i,m_i}\}$
and suppose that \po\ plays $u\in H_i$.
Since $k\ge 2$, there is some $j\not=i$ such that $z_{j,1}\in\cQ_W(u)$,
and so $u$ is not a winning move.
We are done if \pt\ has a winning move, so we assume otherwise and show that \pt\ has a move $v\in\cQ(u)$ that makes $\cQ(u,v)\in\cG$, which will complete the induction.
For ease of notation, relabel the $\{H_i\}$ so that \po\ plays in $H_k$.

First consider the case that $u=x_k$.
If $k=2$ then \pt\ has a winning move $v=x_1$, so we must have $k> 2$.
Thus \pt\ plays $v=z_{k,m_k}$, so that $\cQ(u,v)\in\cG$.
This is because an even number of vertices, namely $\{z_{k,1},\ldots,z_{k,m_k-1}\}$, have been moved into $G_2$.

Next consider the case that $u=z_{k,m_k}$.
If $m_k=1$ and $k=2$, then \pt\ has the winning move $v=x_1$, so we must have either $m_k\ge 2$ or $k\ge 3$.
If $m_k\ge 2$, then \pt\ can play $v=z_{k,m_{k-1}}$, while if $k\ge 3$ then \pt\ can play $v=x_k$.
In either case, we have $\cQ(u,v)\in\cG$.
In the former case, this is because no new isolated vertices were created and the oddness of leaves was maintained.
In the latter case, this is because an even number of shaded leaves were moved to $G_2$.
\hfill$\diamond$

\item
Rewrite the bridging $G=(G_0,x)$--$(G_0,x)^k$ as $(G_0,x_0)$--$(\bG,\bx)$, where $\bG=(G_1,\ldots,G_k)$, $\bx=(x_1,\ldots,$ $x_k)$, and each $G_i\cong G_0$.
Since $k$ is odd, we may pair $G_j$ with $G_{j+1}$ for all even $j<k$, which allows us to extend the identity map on $G_0$ to the automorphism $\f$ of $G$ composed of the natural isomorphisms between each pair $G_j$ and $G_{j+1}$.
Now we define the corresponding mirroring strategy for \pt\ that plays $\f(u)$ for each play $u$ by \po.
We claim that this is a winning strategy.

Write $z_i$ for a singleton in $G_i-x_i$.
Suppose that \po\ wins and let $u$ be a winning move by \po.
That is, there is a position $Q$ on $G$ such that $W_Q\not=\mt$ and $W_{Q(u)}=\mt$.
This means that there is some $w\in W_Q\cap N[u]$. 
Because $\f$ is an isomorphism between $G_j$ and $G_{j+1}$, this yields $\f(w)\in W_Q\cap N[\f(u)]$.
Because $\f(w)\not\in W_{Q(u)}$, it must be that $\f(w)\in N(u)$, and so $u=x_i$, for some $i$.
But this implies that $z_i\in W_Q$ since $z_i$ is a leaf, and so $\f(z_i)\in W_Q$.
Since $\f(z_i)\not\in N(x_i)$ , we arrive at the contradiction $\f(z_i)\in W_{Q(u)}$.
Hence \pt\ wins.
\hfill$\diamond$

\end{enumerate}
\noindent
This completes the proof of Theorem \ref{t:InvoP}.
\qed

\subsection{Proof of Theorem \ref{t:DomGameCycle}}
\label{ss:DomGameCycle}

Recall the statement of Theorem \ref{t:DomGameCycle}: 
For $n\ge 3$, $\zD(C_n)\in\sN$ if and only if $n$ is odd but not equal to 5.
Let $\cD_n=\zD(C_n)$.
Because of the symmetry of $C_n$, for any $u$ we have $\cD_n(u)\cong \cP_{n-1}$, and so $\cD_n\in\sN$ if and only if $\cP_{n-1}\in\sP$.
The result follows from Corollary \ref{c:path}.
\qed

\subsection{Proof of Theorem \ref{t:DomGamePath}}
\label{ss:DomGamePath}

Recall the statement of Theorem \ref{t:DomGamePath}: 
For $n\ge 2$, $\zD(P_n)\in\sN$ if and only if $n$ is odd or $n\in\{2, 6, 8, 10, 12\}$.
Let $\cD_n=\zD(P_n)$ and $\po=\nx$, and write $\{0,1,\ldots,n-1\}$ for the vertex labels of $\cD_n$. 

We first suppose that $n$ is odd. \po\ plays the center vertex $u=(n-1)/2\in \cD_n$. Then $\cD_n(u)\cong  \cA\oplus \cB$ where $\cA\cong \cB$, and thus $\cD_n(u)\in \sP$ by Lemma \ref{l:Iso}. Thus $\cD_n\in \sN$.

At this point we recall the more specific notation $\cP_k^i$ for any path on $k$ vertices with exactly $i$ shaded endpoints.
Thus, $\cD_k$ can also be written as $\cP_k^0$.

Next we suppose that $n$ is even and consider the following cases:
\begin{quote}
\begin{itemize}
    \item[$n=2$:] Here \po\ wins with any move so $\cD_2\in \sN$.
    \item[$n=4$:] Let \po\ play $u\in \cD_4$. By symmetry, we can assume that $u=0$ or $u=1$. In both cases, \pt\ wins by playing vertex 3. Thus $\cD_4\in \sP$.
    \item[$n=6$:] In this case \po\ plays an endpoint $u$, so that $\cD_6(u)\cong\cP_5^1$.
    From this, it is easy to see that, for any move $v$ by \pt\, the at most three white vertices of $\cD_6(u,v)$ are consecutive, and subsequently can be dominated on the next move by \po.
    Hence $\cD_6\in\sN$ (and $\cP_5^1\in\sP$).
    \item[$n=8$:] 
    Now \po\ plays a neighbor $u$ of an endpoint.
    If \pt\ plays one neighbor $v$ of $u$ then \po\ plays the other $w$, so that $\cD_8(u,v,w)\cong \cP_5^1$, which is in \sP\ by the argument in the case $n=6$ above.
    So we assume otherwise.
    If \pt\ plays $v$ so that the at most 3 white vertices of $\cD_8(u,v)$ are consecutive, then \po\ plays to dominate them on the next move.
    The only remaining case is that \pt\ plays the vertex $v$ at distance two from the unshaded endpoint of $\cD_8(u)$.
    Now $\cD_8(u,v)\cong \cP_1^1\oplus\cP_3^2\oplus\cP_2^1$, and \po\ plays the unique vertex $w$ in $\cP_1^1$.
    At this point, note that any move in either component $\cP_3^2$ or $\cP_2^1$ dominates that component, so that, when \pt\ plays in one of them, \po\ plays in the other to win.
    Therefore $\cD_8\in\sN$.
    \item[$n\in\{10,12\}$:]
    For these values of $n$, the case analysis is more extensive, but not insightful.
    We verified that $\cD_n\in\sN$ for each such $n$ by computer.
    \item[$n\ge 14$:] 
    The following strategy can be used by \pt\ to win this game.
    Irrespective of the first two moves of \po, the first two moves of \pt\ should be the endpoints of $\cP_n^0$.
    If \pt\ cannot choose an endpoint because \po\ has already played it, then \pt\ can play any other vertex.
    Thus, after each player has made two moves, we arrive at the even position $\cQ=\cP_{a_1}\oplus\cP_{a_2}\oplus\cP_{a_3}$, where $a_1+a_2+a_3=n-4$ and each $a_i\ge 0$.
    Because each path has at most 2 shaded endpoints, $|W_\cQ|\ge (n-4)-6\ge 4$.
    Hence $\cQ\in \sP$ by Theorem \ref{t:characterization}, and so $\cD_n\in \sP$. 
\end{itemize}
\end{quote}

\section{Other specific graphs}
\label{s:Specifics}

The path $P_n$ can be described as the internally 2-regular tree on $n$ vertices.
A natural extension of this class of trees is the set of internally $r$-regular trees.
(Internally 3-regular trees can be thought of as rooted binary trees with an additional pendant vertex attached to the root.)
One instance in this class is the set of internally $r$-regular caterpillars.
A {\it caterpillar} is a tree such that every vertex not on a longest path $P_k$ (called the {\it spine}) is adjacent to some vertex of $P_k$.
Such a caterpillar is called {\it even} ({\it odd}) if $k$ is {\it even} ({\it odd}).
The bridging construction yields the following theorem.

\begin{thm}
\label{t:Int3RegCatEven}
If $G$ is an even internally $r$-regular caterpillar with $r\ge 3$, then $\zD(G)\in\sP$.
\end{thm}

\begin{proof}
Let $e=\{x_1,x_2\}$ be the middle edge of the spine $P$ of $G$; then $G-e$ is the disjoint union of two isomorphic copies of some tree $T_1\cong T_2$.
Then $G=(T_1,x_1)$--$(T_2,x_2)$ and each $T_i-x_i$ has a singleton component, 
so $\zD(G)\in\sP$ by Theorem \ref{t:InvoP}.\ref{cp:Bridge}.
\end{proof}

\begin{thm}
    If $G$ is the $k$-sunlet graph $S_k$, with $k\geq 3$, then $\zD(G)\in\sP$.
\end{thm}

\begin{proof}
    As noted after the dangling definition, $S_k=(C_k,\bx)\cdot (\bH,\by)$, where $\bH=(K_2,\ldots,K_2)$.
    \label{ref2:scoping}
    Hence $\zD(S_k)\in\sP$ by Theorem \ref{t:InvoP}.\ref{cp:DangleEven}.
\end{proof}

Another interesting class of graphs to consider are the grids $P_k\mathbin{\gbox} P_m$.
Because of the complex pattern of $\zD(P_n)$ winner given in Theorem \ref{t:DomGamePath}, one might guess that no simple pattern describes the winner of grids in general.
However, evidence suggests otherwise.

\begin{thm}
\label{t:OddGrid}
If $km$ is odd then $\zD(P_k\mathbin{\gbox} P_m)\in\sN$.
\end{thm}

\begin{proof}
Define $s=(k-1)/2$ and $t=(m-1)/2$, and use the coordinate system $\{-s,\ldots,0,\ldots,s\}\times\{-t,\ldots,0,\ldots,t\}$ for the vertices of $G=P_k\mathbin{\gbox} P_m$.
Define the involution $\f$ on $G$ by $\f(a,b)=(-a,-b)$.
Then $\f$ is 3-involutionary
on $G-(0,0)$.
Hence, \po\ wins $\zD(G)$ by playing $u=(0,0)$ because $G(u)=G-(0,0)\in\sP$ by Lemma \ref{l:Involution}.
\end{proof}

Notice that the same argument yields that $\zD(P_{k_1}\mathbin{\gbox} P_{k_2}\mathbin{\gbox}\cdots\mathbin{\gbox} P_{k_d})\in\sN$ whenever $k_1k_2\cdots k_d$ is odd.
Similarly, toroidal grids are equally interesting.
As pointed out in Example \ref{e:CycleProduct}, which follows from Theorem \ref{t:InvoP}.\ref{cp:CartProd}, even higher-dimensional products of cycles are in \sP, provided that some cycle has length at least 6.
We record this below.

\begin{thm}
\label{t:CycleProds}
Suppose that $n_1\ge 6$ is even, $3\le n_2\le\cdots\le n_d$ are integers, and $G=C_{n_1}\mathbin{\gbox} C_{n_2}\mathbin{\gbox}\cdots\mathbin{\gbox} C_{n_d}$.
Then $\zD(G)\in\sP$.
\qed
\end{thm}

For $m\ge 2t+1$, the {\it Kneser graph} $K(m,t)$ has vertex set $\binom{[m]}{t}$, with edges between disjoint pairs of vertices.
The {\it Petersen graph} equals $K(5,2)$ and is of great importance in many areas of graph theory.
It is natural to investigate which player wins the games on these graphs.

\begin{thm}
\label{t:Petersen}
For the Petersen graph $P$, we have $\zD(P)\in\sP$.
\end{thm}

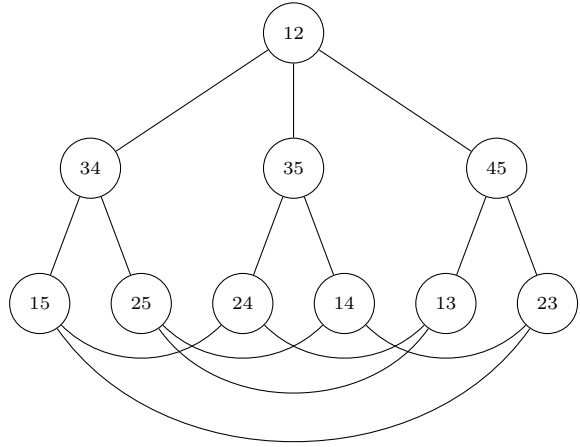
\begin{figure}
\centering
\begin{tikzpicture}[scale=0.45]
\tikzstyle{every node}=[draw,shape=circle,minimum size=.8cm,font=\scriptsize];
\def \wid {3cm}
\def \hit {4cm}
\def \Ang {360/5}
\def \ang {\Ang/4}
\def \radv {2cm}
\def \radz {4.2cm}
\path ({\wid * 2.5},{\hit * 2}) node (12) {$12$};
\path ({\wid * 0.5},{\hit * 1}) node (34) {$34$};
\path ({\wid * 2.5},{\hit * 1}) node (35) {$35$};
\path ({\wid * 4.5},{\hit * 1}) node (45) {$45$};
\path ({\wid * 0},{\hit * 0}) node (15) {$15$};
\path ({\wid * 1},{\hit * 0}) node (25) {$25$};
\path ({\wid * 2},{\hit * 0}) node (24) {$24$};
\path ({\wid * 3},{\hit * 0}) node (14) {$14$};
\path ({\wid * 4},{\hit * 0}) node (13) {$13$};
\path ({\wid * 5},{\hit * 0}) node (23) {$23$};
\draw (12) -- (34);
\draw (12) -- (35);
\draw (12) -- (45);
\draw (34) -- (15);
\draw (34) -- (25);
\draw (35) -- (14);
\draw (35) -- (24);
\draw (45) -- (13);
\draw (45) -- (23);
\draw (15) to[out=-45,in=-135] (24);
\draw (25) to[out=-45,in=-135] (14);
\draw (24) to[out=-45,in=-135] (13);
\draw (14) to[out=-45,in=-135] (23);
\draw (25) to[out=-55,in=-125] (13);
\draw (15) to[out=-55,in=-125] (23);
\end{tikzpicture}
\caption{The Petersen graph $P=K(5,2)$ with its vertex labels shown.}
\label{fig:Pete}
\end{figure}

\begin{proof}
We simplify the notation of the vertices of $P$ by writing $ij$ instead of $\{i,j\}$ (see Figure \ref{fig:Pete}).
By symmetry, we may assume that \po\ chooses 12.
Because $P-N[12]$ is a 6-cycle $C$, which is 3-involutionary, \pt\ will lose by choosing any vertex $u$ of $C$ (\po\ responds by playing the vertex of $C$ opposite from $u$).
By symmetry again, we may assume that \pt\ chooses 35.
Similarly, \po\ loses by choosing a vertex of $C$, and so chooses, by symmetry, 34.
But now \pt\ wins by choosing 45.
\end{proof}

\section{Final Comments}  
\label{s:Comments}

Given the results of Section \ref{s:Specifics}, we offer the following open problems and conjectures.

\begin{prb}
\label{p:OddCat}
Find the winning player for $\zD(G)$ when $G$ is an odd internally regular caterpillar.
\end{prb}

The first several cases of the following conjecture are easy to verify by hand.

\begin{cnj}
\label{c:PInvo}
For $k, m > 1$ with $mk$ even, we have $\zD(P_k\mathbin{\gbox} P_m)\in\sP$.
\end{cnj}

\begin{prb}
\label{p:Tori}
For integers $3\le n_1\le\ldots\le n_d$, with every even $n_i=4$, determine the winner of $\zD(G)$ for $G=C_{n_1}\mathbin{\gbox}\cdots\mathbin{\gbox} C_{n_d}$.
\end{prb}

A general statement one might hope to prove about Cartesian products involves the case in which the winner of one of the graphs is known.

\begin{qst}
\label{q:PInvo}
Let $G, H$ be graphs such that $\zD(G)\in \sP$. 
Under what conditions is it true that $G\mathbin{\gbox} H\in \sP$?
\end{qst}

For example, Theorem \ref{t:InvoP}.\ref{c:PInvo} states that $\zD(G\mathbin{\gbox} H)\in\sP$ if $G$ is 3-involutionary, so the question is open for graphs in \sP\ that are not 3-involutionary, such as $\zD(P_4)$.
If Conjecture \ref{c:PInvo} is true, though, $\zD(P_4\mathbin{\gbox} P_m)\in\sP$ for any $m\ge 2$.

A similar question arises when considering graph powers, along the lines of Theorem \ref{t:InvoP}.\ref{cp:Cpower}.

\begin{qst}
\label{q:GPowers}
For what conditions on a graph $G$ can we determine the winner of $\zD(G^{(k)})?$
\end{qst}

We listed five graph constructions in Section \ref{ss:GraphCon} that produce \sP\ games by Theorem \ref{t:InvoP}.
There are many other graph constructions one might consider.

\begin{qst}
\label{q:GraphCons}
What other graph constructions produce games in \sP?
\end{qst}

Finally, we offer a reachable problem on problem on Kneser graphs.

\begin{prb}
\label{p:Kneser}
Find the winning player for $\zD(G)$ when $G$ is the Kneser graph $K(m,2)$ and $m\ge 6$.
\end{prb}

While solving this game on a general graph appears to be out of reach, the game can be played on a number of other classes of graphs that we believe to be natural directions for future study. 
Some of the techniques described above may shed light on the case of trees. 
In particular, spiders, other caterpillars, and other internally 3-regular trees seem to be approachable cases to work on in the future; complete bipartite graphs are also a natural choice, as are non-Cartesian products of graphs and higher order Kneser graphs.

Finally, there are several related variants of the graph domination game.
The first is the Mis\`ere version in which the first player to create a dominating set loses rather than wins.
Another version that can be played is that players take turns choosing vertices until the complement of the set of chosen vertices is not a dominating set. Additionally, all of these variants can be played with ``dominating sets'' replaced by ``total dominating sets'', which are subsets $D\subseteq V$ of vertices of a graph with the property that $N(D)=V$.


\section*{Acknowledgement}

This material is based upon work supported by the U.S. National Science Foundation under Grant No. 2015425.
The authors are grateful to the anonymous referees for excellent suggestions that improved the exposition in this paper.


\bibliographystyle{acm}
\bibliography{refs}


\section*{Statements and Declarations}

The authors have no relevant financial or non-financial interests to disclose and no conflicts of interest to report.
No data was involved in this research.

\end{document}